\documentclass{article}
\usepackage[utf8]{inputenc}
\usepackage{color}
\usepackage{xcolor}

\usepackage{amsmath,amsfonts,euscript,amssymb,amsthm,a4,times}

\usepackage{eurosym}
\usepackage{amssymb}
\usepackage{amstext}
\usepackage{color}
\usepackage{amsmath,amsthm,amssymb,amscd,epsfig,esint, mathrsfs,graphics,color}
\usepackage{wrapfig}
\usepackage{verbatim}

\newcommand{\loc}{\textnormal{loc}}
\newcommand{\medint}{-\kern  -,375cm\int}

\newenvironment{michelarev}{\color{red}}{\color{black}}
\newcommand{\bmicr}{\begin{michelarev}}
	\newcommand{\emicr}{\end{michelarev}}

\def\cA{\mathcal{A}}
\def\R{\mathbb{R}}

\allowdisplaybreaks

\def\XXint#1#2#3{{\setbox0=\hbox{$#1{#2#3}{\int}$}
\vcenter{\hbox{$#2#3$}}\kern-.5\wd0}}

\newtheorem{theorem}{Theorem}[section]

\newtheorem{lemma}[theorem]{Lemma}

\theoremstyle{definition}
\newtheorem{definition}[theorem]{Definition}
\newtheorem{remark}[theorem]{Remark}

\makeatletter \catcode`@=11
\newbox\tr@tto
\setbox\tr@tto=\hbox{{\count0=0\dimen0=-,9pt\dimen1=1,1pt\loop\ifnum
    \count0<11 \advance \count0 by1 \vrule width.51pt height\dimen1
    depth\dimen0\kern-0.17pt\advance\dimen0 by-0.05pt\advance\dimen1
    by0.1pt\repeat \loop\ifnum\count0<21\advance \count0 by1 \vrule
    width.6pt height\dimen1 depth\dimen0\kern-0.2pt \advance\dimen0
    by-0.1pt\advance\dimen1 by 0.05pt\repeat}}
\def\medint{\displaystyle\copy\tr@tto\kern-10.4pt\int}
\numberwithin{equation}{section}

\def\R{{\mathbb R}}

\def\proofof#1{\begin{proof}[Proof of #1]}

\def\dx{{\mathrm d}x}

\def\loc{\rm loc}

\catcode`@=11
\newbox\tr@tto
\setbox\tr@tto=\hbox{{\count0=0\dimen0=-,9pt\dimen1=1,1pt\loop\ifnum\count0<11 \advance \count0 by1 \vrule width.51pt height\dimen1
     depth\dimen0\kern-0.17pt\advance\dimen0 by-0.05pt\advance\dimen1
     by0.1pt\repeat \loop\ifnum\count0<21\advance \count0 by1 \vrule
     width.6pt height\dimen1 depth\dimen0\kern-0.2pt \advance\dimen0
     by-0.1pt\advance\dimen1 by 0.05pt\repeat}}
\def\medint{\displaystyle\copy\tr@tto\kern-10.4pt\int}
\catcode`@=12

\newcommand{\LL}{\mathrm{L}}

\numberwithin{equation}{section}

\begin{document}
\title{Higher differentiability of solutions for a class of obstacle problems with variable exponents}
\author{Niccol\`o Foralli, Giovanni Giliberti}

\maketitle

\begin{abstract}
In this paper we prove a higher differentiability result for the solutions to a class of obstacle problems in the form
 \begin{equation*}
\label{obst-def0}
\min\left\{\int_\Omega F(x,Dw) dx : w\in \mathcal{K}_{\psi}(\Omega)\right\}
\end{equation*}
where $\psi\in W^{1,p(x)}(\Omega)$ is a fixed function called obstacle  and $\mathcal{K}_{\psi}=\{w \in W^{1,p(x)}_{0}(\Omega)+u_0: w \ge \psi \,\, \textnormal{a.e. in $\Omega$}\}$ is the class of the admissible functions, for a suitable boundary value $ u_0 $. We deal with a convex integrand $F$ which satisfies the $p(x)$-growth conditions
\begin{equation*}\label{growth}|\xi|^{p(x)}\le F(x,\xi)\le C(1+|\xi|^{p(x)}),\quad p(x)>1
\end{equation*}

\end{abstract}

\maketitle

\noindent
{\footnotesize {\bf AMS Classifications.}  49N15; 49N60; 49N99.}

\noindent
{\footnotesize {\bf Key words and phrases.}  Variational integrals; Variable exponents; Regularity of minimizers; Higher differentiability.}

\bigskip

\bigskip
\section{Introduction}
\bigskip

The aim of this paper is to study higher differentiability results for the solutions to the following class of obstacle problems
$$ \min\left\{\int_\Omega F(x,Dw)\,dx : w\in\mathcal{K}_\psi(\Omega)\right\} $$
where $ \psi\in W^{1,p(x)}(\Omega) $ is a fixed function, called obstacle and
$$ \mathcal{K}_\psi(\Omega)=\left\{w\in W^{1,p(x)}_0(\Omega)+u_0:w\ge \psi \text{ a.e. in } \Omega\right\} $$
is the class of admissible functions, for a suitable boundary value $ u_0 $.\\
Here $ \Omega $ is a bounded open set of $ \mathbb{R}^n $, $ n>2 $ and $ F:\Omega \times \mathbb{R}^n \rightarrow \mathbb{R}$ is a Carathéodory function fulfilling natural growth and convexity assumptions with variable exponent (namely assumptions ($\mathcal{A}1$)--($\mathcal{A}3$) below).\\
Higher diferentiability results have been attracting a lot of attention in the recent years, starting from the pioneering papers \cite{PdN14-1,PdN14-2,PdN15,G15,GPdN,KM}, to the more recent results concerning higher differentiability results for obstacle problems in the case of standard growth conditions \cite{EPdN18, EPdN20, Gr21} and $ p-q $ growth conditions of integer \cite{Gav1, Gav2, CDF20} and fractional order \cite{GrIp}, see also \cite{CDF20}, including the case of nearly linear growth \cite{Gav3} and the subquadratic growth case \cite{Ge21}.\\
In the same spirit of these results, assuming that the gradient of the obstacle belongs to a suitable Sobolev class, we are interested in finding conditions on the partial map 
$$ x \longmapsto \mathcal{A}(x,\xi):=D_\xi F(x,\xi) $$

in order to obtain that the extra differentiability property of the obstacle transfers to the gradient of the solution, possibly with no loss in the order of differentiability.\\
It is indeed well known, also for equations in divergence form, that no extra-differentiability properties for the solutions can be expected even if the right hand side is smooth, unless some assumptions are given on the $ x $-dependence on $ \mathcal{A} $.\\
Here the situation is much more complicated because we are considering the framework of variable exponent setting, a topic to which researchers devoted a lot of attention in the last decades from different viewpoints \cite{Di, Min06}.\\
Up to our knowledge, the only higher differentiability result available in this framework is the one in \cite{GPdN15}, in the case of unconstrained minimizers. In this case the assumptions on the partial map keep into account the weak differentiability of the variable exponent, more precisely the fact that its weak gradient belongs to some $ L \, \log^\beta L $ space, for some suitable $ \beta $.\\
Inspired by this result, we assume that the map $ x \longmapsto \mathcal{A}(x,\xi) $ is weakly differentiable in a suitable sense, see assumption ($\mathcal{A}4$) below, by means of a pointwise characterization of Sobolev spaces due to Haj\l asz, see \cite{H96}.\\
However our situation turns out to be even more complicated for two reasons:\\
1)	on one hand, differently from \cite{GPdN15}, we cannot deduce the higher differentiability estimates by simply relying on a class of approximating auxiliary problems and then passing the estimates to the limit, because in the constrained case we are not allowed to use an appropriate choice of the test functions and, along the recent Lipschitz regularity results (see for istance \cite{CEP}) we would need to use a linearization technique.
At this point we prefer to deal with the difference quotient method, so that assumption ($\mathcal{A}4$) turns to be the appropriate assumption (see for instance \cite{AM02} for the corresponding hypothesis in the case of H\"older continuity results);\\
2)	on the other hand, one of the main tools we need in order to perform our higher differentiability result is a Calderon-Zygmund regularity result, more precisely \cite{EH08}, see also \cite{BCO16, Mingione}.
This necessarily requires a more refined quantitative assumptions on the function $ p $ with respect to the log-Holder continuity introduced by Zhikov in \cite{Zhik}.
This assumption turns to be very natural for instance in the context of H\"older continuity results for either unconstrained and constrained problems \cite{AM01, AM02, Ele04, EH08}.
In our case, in view of recent embedding theorems in the context of Orlicz-Sobolev setting (see \cite{Cianchi}) we have that a $ W^{1,1} $ function $ k $ such that $ |Dk| \in L^n \log^{\sigma}L(\Omega) $ for some $ \sigma $ sufficiently large, will be actually a continuous function, with a modulus of continuity $ \omega $ such that 
$$ \lim_{ R \rightarrow 0} \omega(R) \log\left(\frac{1}{R}\right) = 0.  $$
Therefore the above mentioned Calderon-Zygmund regularity results apply in our situation.\\
We would like to stress that our assumption ($ \mathcal{A}4 $) turns to be very general and include not only the model case
$$ \mathcal{F}(u) = \int_\Omega|Du|^{p(x)} $$
 but also more general situations like 
 $$ \mathcal{G}(u) = \int_\Omega a(x)h(|Du|)^{p(x)}) $$ 
 recently considered (from the viewpoint of Lipschitz continuity) in \cite{EMM16}, a result which in turn has been very recently generalized in \cite{EPdN21}.\\
The paper therefore is organized as follows: Section 2 contains the notations and the statement of our main theorem; Section 3 includes some preliminary results concerning spaces with variable exponents, difference quotient and the above mentioned Calder\'on-Zygmund theorem; Section 4 is instead devoted to the proof of the main result.

\section{Notations and statement of the main result}

As mentioned in the introduction, we are going to prove a higher differentiability property of the gradient of the solutions $u\in W^{1,p(x)}(\Omega)$ to  variational obstacle problems of the form
\begin{equation}
\label{obst-def0}
\min\left\{\int_\Omega F(x,Dw): w\in \mathcal{K}_{\psi}(\Omega)\right\}
\end{equation}
where the integrand $F : \, \Omega \times \mathbb{R}^n \rightarrow \mathbb{R}  $ is a convex Carathéodory function of class $ C^1 $.\\
Here $ \Omega $ is a bounded open subset of $ \mathbb{R}^n $, $\psi:\,\Omega \rightarrow [- \infty, + \infty)$ is a fixed function called \textit{obstacle} which belongs to the Sobolev space $W^{1,p(x)}(\Omega)$.\\
The admissible functions class $\mathcal{K}_{\psi}(\Omega)$ is defined as
\begin{equation}
\label{classeA}
\mathcal{K}_{\psi}(\Omega) := \left \{w \in W^{1,p(x)}_{0}(\Omega)+u_0: w \ge \psi \,\, \textnormal{a.e. in $\Omega$} \right\}
\end{equation}
where $ u_0\in W^{1,p(x)}_{\loc}(\Omega) $ is a suitable boundary value.

By replacing $u_0$ by $\tilde{u}_0 = \max \{u_0, \psi\}$, we
may assume that the boundary value function $u_0$ satisfies $u_0\ge \psi$.

Therefore the set $\mathcal{K}_{\psi}$ is not empty.

We will deal with local solutions to our obstacle problem, in the following sense:
\begin{definition}\label{probdef}
	A mapping $u\in \mathcal{K}_\psi(\Omega)$ is a {\it local solution to the obstacle problem in $\mathcal{K}_{\psi}(\Omega)$}  if
	$F(x,Du) \in \LL^{1}_{\rm loc}(\Omega )$ and
	$$
	\int_{\text{supp}(u-\varphi)} \! F(x,Du) \, \dx \leq \int_{\text{supp}(u-\varphi)} \! F(x,D\varphi) \, \dx
	$$
	for  any $\varphi\in \mathcal{K}_{\psi}(\Omega)$.
\end{definition}

 We recall that    $u \in W^{1,p(x)}(\Omega)$ is a {\it local solution to the obstacle problem in $\mathcal{K}_{\psi}(\Omega)$}  in the spirit of Definition \ref{probdef} if and only if $u\in \mathcal{K}_\psi$ solves the following variational inequality
\begin{equation}
\label{obst-def}
\int_{\Omega} \langle \mathcal{A}(x, Du), D(\varphi - u) \rangle \, dx \ge 0,
\end{equation}
for all  $\varphi \in \mathcal{K}_{\psi}(\Omega)$ with supp$\varphi\Subset\Omega$,  where we set $$\mathcal{A}(x,\xi)=D_\xi F(x,\xi).$$
Therefore, for the sake of semplicity, from now on we will give our assumptions on $ \mathcal{A} $ instead of $ F $.\\
More precisely in the sequel we will assume that there exist a continuous function $p:\,\Omega \rightarrow (1, + \infty)$  with
\begin{equation}
	\label{boundp}
	1< \gamma_1 := \inf_{\Omega}p(x)\le p(x) \le \gamma_2 := \sup_{\Omega}p(x)<n<+\infty 
\end{equation}
for some $ \gamma_1, \gamma_2 $ and for every $ x\in\Omega $ and there exist positive constants $\nu, L$ and $\ell$ such that  the following $p(x)$-ellipticity and $p(x)$-growth conditions are satisfied for a.e. $x\in \Omega$ and every $\xi,\eta\in \mathbb{R}^n$:
$$ \langle \mathcal{A}(x, \xi) - \mathcal{A}(x, \eta), \xi - \eta \rangle  \ge \, \nu |\xi - \eta|^2 (1 + |\xi|^2 + |\eta|^2)^{\frac{p(x)-2}{2}}\eqno{(\cA 1)}$$
$$ |\mathcal{A}(x, \xi) - \mathcal{A}(x, \eta)| \le \, L \, |\xi - \eta| (1 + |\xi|^2 + |\eta|^2)^{\frac{p(x)-2}{2}} \eqno{(\cA 2)}$$
$$ |\mathcal{A}(x, \xi)| \le \, \ell \, (1 + |\xi|^2)^{\frac{p(x)-1}{2}} \eqno{(\cA 3)}$$

We also assume that there exists a non negative function $\kappa \in L^n\log^\sigma L(\Omega)$ with $ \sigma>2n-1$ such that $\mathcal{A}$ satisfies 
$$|\mathcal{A}(x,\xi)-\mathcal{A}(y,\xi)| \le (\kappa(x)+\kappa(y))|x-y|(1+|\xi|^2)^{\frac{p(x)-1}{2}}\log(e + |\xi|^2) \eqno{(\cA 4)}$$

\bigskip
See Section \ref{logspace} for more details about the spaces $ L^n\log^\sigma L(\Omega) $. 

Under these assumptions our higher differentiability result is the following

\begin{theorem}\label{teorema}
	Let $\mathcal{A}(x, \xi)$ satisfy $(\cA 1)$--$(\cA4)$, $ p:\Omega\rightarrow(1,+\infty) $ a continuous function satisfying \eqref{boundp} with $\gamma_1>2$.\\
	Let $u \in \mathcal{K}_{\psi}(\Omega)$ be the  solution to the obstacle problem \eqref{obst-def}. Then we have
	\begin{equation}
	\label{tesi2}
	D \psi \in W^{1,\gamma_2}_{\rm loc}(\Omega) \Rightarrow (1 + |D u|^2)^{\frac{\gamma_1-2}{4}} Du \in W^{1,2}_{\rm loc}(\Omega)
	\end{equation}
\end{theorem}

\begin{remark}
	The assumption on $ \kappa $ in ($ \mathcal{A}4 $) yields the corresponding regularity for $ p $, see Theorem \ref{kp} and its consequences.
\end{remark}
Existence of solutions to the obstacle problem \eqref{obst-def0} can be easily proved through classical results regarding variational inequalities, so in this paper we will mainly concentrate on the regularity results. The main point will be the choice of  suitable test functions $\varphi$ in \eqref{obst-def} that involve the difference quotient of the solution but at the same time turns to be admissible for the obstacle class $\mathcal{K}_{\psi}(\Omega)$.

\section{Notations and preliminary results}

In this paper we denote by $ C $ a positive constant, that could vary in each line. We highlight in brackets the dependence on relevant parameters when needed.\\
We indicate with $ B(x_0,r)=B_r(x_0) =\left\{x\in\mathbb{R}^n:|x-x_0|<r\right\} $ the ball of center $ x_0 $ and radius $ r $, we omit the dependence on the  center and the radius if not necessary.\\
Let us introduce the following auxiliary function
\begin{equation}
V_p(\xi):=\left(1+|\xi|^2\right)^{\frac{p-2}{4}}\xi
\end{equation}
defined for all $ \xi \in \mathbb{R}^n $.

\subsection{Some elementary inequalities}
For the function $ V_p $ we recall the following result (for the proof see \cite{Giusti}, Lemma 8.3)
\begin{lemma} \label{Vi}
	Let $1<p<\infty$. There exists a constant $C=C(n,p)>0$
	such that
	$$
	C^{-1}\Bigl( 1+| \xi |^2+| \eta |^2 \Bigr)^{\frac{p-2}{2}}\leq
	\frac{|V_{p}(\xi )-V_{p}(\eta )|^2}{|\xi -\eta |^2} \leq
	C\Bigl( 1 +|\xi |^2+|\eta |^2 \Bigr)^{\frac{p-2}{2}}
	$$
	for any $\xi$, $\eta \in \R^n$.
\end{lemma}

In the sequel we will often use also the inequality (see \cite{GPdN15})
\begin{lemma} \label{dislog}
	For every $s,t>0$ and for every $\alpha,\varepsilon,\gamma>0$ we have
	$$st\le\varepsilon s\log^\alpha(e+\gamma s)+\frac{t}{\gamma}\left[\exp\left(\frac{t}{\varepsilon}\right)^\frac{1}{\alpha}-1\right] $$ 
\end{lemma}

\subsection{Spaces with variable exponents}

\begin{definition}
	Let $ \Omega $ be a bounded subset of $ \mathbb{R}^n $. \\
	Let $ p:\Omega\rightarrow (1,+\infty) $ be a continuous function, the space $ L^{p(\cdot)}(\Omega,\mathbb{R}) $ is defined as
	$$ L^{p(\cdot)}(\Omega,\mathbb{R}):=\left\{f:\Omega \rightarrow \mathbb{R} : f \text{ is measurable and} \int_{\Omega}|f(x)|^{p(x) }dx< +\infty\right\} $$
	and, equipped with the Luxemburg norm
	$$ ||f||_{L^{p(\cdot)}(\Omega,\mathbb{R})} := \text{inf}\left\{\lambda>0 : \int_\Omega\left|\frac{f(x)}{\lambda}\right|^{p(x)}dx \le 1\right\}$$
	it becomes a separable Banach space.
\end{definition}

Then, we recall the definition of Sobolev space with variable exponents.
\begin{definition}
	The space $ W^{1,p(x)}(\Omega,\mathbb{R})$ is defined as
	$$ W^{1,p(x)}(\Omega,\mathbb{R}):=\left\{f \in L^{p(x)}(\Omega,\mathbb{R}) : Df \in L^{p(x)}(\Omega,\mathbb{R}^n)\right\} $$
	where $ Df $ denotes the gradient of $ f $. This space becomes a Banach space equipped with the norm:
	$$ ||f||_{ W^{1,p(x)}(\Omega,\mathbb{R})}:= ||f||_{L^{p(x)}(\Omega,\mathbb{R})} + ||Df||_{L^{p(x)}(\Omega,\mathbb{R}^n)} $$
\end{definition}
For more details concerning spaces with variable exponents see \cite{Di}.

\subsection{Difference quotient}
\medskip
\noindent
For every   function $F:\mathbb{R}^{n}\to\mathbb{R}$
the finite difference operator is defined as
$$
\tau_{s,h}F(x):=F(x+he_{s})-F(x)
$$
where $h\in\mathbb{R}$, $e_{s}$ is the unit vector in the $x_{s}$
direction and $s\in\{1,\ldots,n\}$.

	

\noindent The following result concerning finite difference operator is a kind of integral version
of Lagrange Theorem.
\begin{lemma}\label{le1} If \space $0<\rho<R$, $|h|<\frac{R-\rho}{2}$, $1 < p <+\infty$,
	and $F, DF\in L^{p}(B_{R})$ then
	$$
	\int_{B_{\rho}} |\tau_{h} F(x)|^{p}\ dx\leq c(n,p)|h|^{p} \int_{B_{R}}
	|D F(x)|^{p}\ dx .
	$$
	Moreover
	$$
	\int_{B_{\rho}} |F(x+h )|^{p}\ dx\leq  \int_{B_{R}}
	|F(x)|^{p}\ dx .
	$$
\end{lemma}

\bigskip

\noindent Now let us recall the fundamental Sobolev embedding property.
\begin{lemma}\label{lep} Let $F:\mathbb{R}^{n}\to\mathbb{R}$, $F\in
	L^{p}(B_{R})$ with $1<p<n$. Suppose that there exist $\rho\in(0,R)$ and  $M>0$ such that
	$$
	\sum_{s=1}^{n}\int_{B_{\rho}}|\tau_{s,h}F(x)|^{p}\,dx\leq
	M^{p} |h|^{p},
	$$
	for every $h$ with $|h|<\frac{R-\rho}{2}$. Then $F\in
	W^{1,p}(B_{\rho})\cap L^{\frac{np}{n-p}}(B_{\rho})$. Moreover
	$$
	||DF||_{L^{p}(B_{\rho})}\leq
	M
	$$
	and
	$$
	||F||_{L^{\frac{np}{n-p}}(B_{\rho})}\leq
	c\left(M+||F||_{L^{p}(B_{R})}\right),
	$$
	with $c= c(n,p, \rho,R)$.
\end{lemma}
For the proof see, for example, \cite[Lemma 8.2]{Giusti}.

\bigskip

\subsection{Calder\'on-Zygmund estimates}

We will need in the sequel to apply this result that can be found in \cite{EH10} and that can be applied because our assumptions on $ \mathcal{A} $ entail tha corresponding ones on $ F $ and moreover because the assumption ($ \mathcal{A}4 $) on $ \kappa $ entails the corresponding regularity for $ p $, see Theorem \ref{kp} and remark \ref{modp}. The original result is stated for cubes, but, by suitable modifications, it can be stated for balls.

\begin{theorem}
	\label{habermann10}
	Let $ u\in W^{1,p(x)} $ be a solution to the obstacle problem \eqref{obst-def}, where $ \mathcal{A} $ satisfies the assumptions $(\mathcal{A}1)-(\mathcal{A}4)$ and where $ \psi $ is a given obstacle function which satisfies
	$$ |D\psi|^{p(x)}\in L^q_{\loc} (\Omega)$$
	for some $ q>1 $. Then $ |Du|^{p(x)}\in L^q_{\loc} (\Omega) $.\\
	In particular there holds: if  $\, \Omega' \Subset\Omega$ and $ |D\psi|^{p(x)}\in L^q (\Omega') $, then for any given $ \varepsilon\in(0,q-1) $ there exists a positive radius $ R_0>0 $, depending on $ n,\gamma_1, \gamma_2, \nu, L, \varepsilon, q, \Vert|Du|^{p(x)}\Vert_{L^1(\Omega')},\Vert|D\psi|^{p(x)}\Vert_{L^q(\Omega')} $, such that for any ball $ B_{8R}\Subset\Omega' $ and $ R\le R_0 $ there holds
	$$ \left[\fint_{B_R}|Du|^{p(x)q}\,dx\right]^{1/q}\le C K^\varepsilon \fint_{B_{8R}}|Du|^{p(x)}\,dx + C K^\varepsilon \left[\fint_{B_{8R}}|D\psi|^{p(x)q}\,dx+1\right]^{1/q}$$
	where $ C=C(n,\gamma_1,\gamma_2,\nu,L,q) $ and
	$$ K:=\int_{B_{8R}}(|Du|^{p(x)}+|D\psi|^{p(x)(1+\varepsilon)})\,dx+1 $$
\end{theorem}

\begin{remark}
	As a consequence of our result of Calder\'on-Zygmund type we can deduce the clasical higher integrability result, like \cite[Lemma 3.2]{EH08}, so in particular there exist two positive constants $c_0,\delta$ such that 
	\begin{equation}
	\label{magg-int}
	 \left[\fint_{B_{R/2}}|Du|^{p(x)(1+\delta)}\,dx\right]^{1/(1+\delta)}\le c_0 \fint_{B_R}|Du|^{p(x)}\,dx + c_0 \left[\fint_{B_R}(|D\psi|^{p(x)(1+\delta)}+1)\,dx\right]^{1/(1+\delta)}
   \end{equation}
\end{remark}

\subsection{Spaces  $L^p \log^\sigma L$ }\label{logspace}

\begin{definition}
	The space $ L^p\log^\sigma L(\Omega,\mathbb{R}) $ is defined for every $ p\ge 1 $ and $ \sigma \in \mathbb{R} $ as 
	$$  L^p\log^\sigma L(\Omega,\mathbb{R}) :=\left\{f:\Omega \rightarrow \mathbb{R} : f \text{ is measurable and} \int_\Omega |f|^p \, \log^\sigma (e + |f|)dx < +\infty\right\} $$
	and equipped with the Luxemburg norm
	$$ ||f||_{L^p\log^\sigma L(\Omega,\mathbb{R})}:=\inf\left\{\lambda>0 : \int_\Omega \left|\frac{f}{\lambda}\right|^p \, \log^\sigma\left(e + \frac{|f|}{\lambda}\right)dx\le 1\right\} $$
	it becomes a Banach space.
\end{definition}	
	Set
	$$ [f]_{L^p\log^\sigma L(\Omega,\mathbb{R})} = \left( \int_\Omega|f|^p\, \log^\sigma \left(e + \frac{|f|}{||f||_p}\right)dx\right)^\frac{1}{p}$$
	We recall that, for $ p\ge1 $ and $ \sigma \ge 0 $, there exist two positive constant $ \lambda=\lambda(p,\sigma) $ and $ \Lambda=\Lambda(p,\sigma) $ such that
	$$ \lambda [f]_{L^p\log^\sigma L(\Omega,\mathbb{R})} \le ||f||_{L^p\log^\sigma L(\Omega,\mathbb{R})}  \le \Lambda[f]_{L^p\log^\sigma L(\Omega,\mathbb{R})} $$
 We recall the following fact which can be found in \cite{Iw92, IwVe}, see also \cite{AM05}.
	\begin{lemma}\label{normDu}
		For any $p>1$, $h \in L^p(B_R)$ and $\alpha>1$, it holds
		$$\fint_{B_R} |h| \log^\alpha \left(e+\frac{|h|}{\|h\|_1}\right)dx\le c(\alpha,p) \left( \fint_{B_R} |h|^p dx\right)^{1/p}$$
		where, for any $p\ge1$ 	$$\|h\|_p:=\left(\fint_\Omega |h|^p dx\right)^{1/p}.$$
	\end{lemma}
	Next, we recall an embedding Theorem in the Orlicz-Sobolev setting, for more details see \cite{Cianchi}. 
	
	\begin{theorem} \label{kp}
		Let $ \kappa \in W^{1,1}(\Omega) $ be a function such that $ |D\kappa|\in L^n\,\log^\sigma L(\Omega) $, for some $ \sigma > n-1 $. Then $ \kappa \in C^0(\Omega) $ and
		$$ |\kappa(x)-\kappa(y)|\le\frac{c_n}{\log\left(e+\frac{1}{|x-y|}\right)^{\frac{\sigma-n+1}{n}}}||D\kappa||_{L^n\log^\sigma L(\Omega)} $$
	\end{theorem}
\begin{remark}\label{modp}
Observe that the right hand side represents a modulus of continuity  $ \omega(\cdot) $ for the function $ \kappa $. Hence if $ \sigma>2n-1 $ then it has the property
\begin{equation} \label{modcont}
\lim_{R\rightarrow 0} \omega (R) \log\left(\frac{1}{R}\right)=0 
\end{equation}
By considering $y=x_0 $, assumption $(\mathcal{A}4)$ entails that the same regularity for $\kappa$ can be transferred to the exponent $p$, which therefore is continuous, with modulus of continuity fulfilling \eqref{modcont}. In the sequel we will frequently use this assumption. In particular this turns to be necessary in order to apply Theorem \ref{habermann10}.

\end{remark}

\section{Proof of Theorem \ref{teorema}}

Since all our results are local in nature, without loss of generality we shall suppose that
$$ \int_\Omega|Du|^{p(x)}dx < +\infty$$
Observe that since we have $|D\psi|\in W^{1,\gamma_2}(\Omega)$, due to Sobolev embedding theorem $|D\psi|\in L^{\gamma_2^*}(\Omega)$, that is 
$$ \int_\Omega|D\psi|^{\gamma_2^*}dx < +\infty\quad\quad \text{where}\quad\quad\gamma_2^*:=\frac{n\gamma_2}{n-\gamma_2}$$
We can define the finite quantity
\begin{equation} \label{finfin}
    M=\int_\Omega(|Du|^{p(x)}+|D\psi|^{\gamma_2^*})\,dx
\end{equation}

We start by applying the higher integrability result \eqref{magg-int}.
We would like to stress the fact that the higher integrability constants $ c_0, \delta $ are independent of the function $ F $ and also of the minimizer $ u $; they only depend on the growth constants and on the quantity $ M $ defined above. So, once the quantities $ \gamma_1, \gamma_2, L, \ell, \nu, M $ are fixed, the constants are determined independently of the function $ F $ and of the minimizer $ u $.
Of course $ \delta $ can be replaced at will by smaller constants.

	

Once obtained the higher integrability $ \delta $, we select $  R_0 \equiv R_0(\gamma_1 \, \gamma_2, L, \ell, \nu, M)>0 $ with the property that $ \omega (16R_0) <\delta/4 $, where $\omega :\mathbb{R}^+ \rightarrow \mathbb{R}^+$ is a nondecreasing continuous function, vanishing at zero, which represents the modulus of continuity of the exponent $ p(x)$ in view of remark \ref{modp}.\\
Finally we fix a ball $ B_{R_0} \subset\subset \Omega  $ and we set
$$ p_m:= \underset{\overline{{B_{R_0}}}}{\max}\,\,\,  p(x) $$
Then we consider balls $  B(x_c, 8R)\equiv B_{8R}\subset\subset B_{R_0/4} $ and we define 
$$p^+:=\underset{\overline{{B_{8R}}}}{\max}\,\,\,  p(x) \,\,\,\,\,\,\,
p^-:=\underset{\overline{{B_{8R}}}}{\min}\,\,\,  p(x)$$
Note that $ p^+ $ and $  p^- $ depend on the ball. Note also that, for a suitable $ x_0 \in  \overline{B_{8R}} $, not necessarly the center, we have $ p^+=p(x_0) $; also we have $ p^+-p^- \le \omega (16R) \le 16\omega (R) $.
Arguing as in \cite{AM01}, the preceding choices imply that
$$p^+(1+\delta/4)\le p(x)(1+\delta)\,\,\,\,\, \text{in } B_{8R}$$
$$p_m(1+\delta/4)\le p(x)(1+\delta) \,\,\,\,\,\text{in } B_{R_0}$$
Finally, without loss of generality, we can always choose $  16R \le R_0 \le 1 $.\\
For brevity, we state this Lemma which will be frequently used in the sequel:
 
 \begin{lemma}
 	\label{maggint}
 	Let $ u\in \mathcal{K}_\psi(\Omega) $ be a solution to the obstacle problem \eqref{obst-def} and the quantities $ p^+, R_0, \delta $ defined above, for any $\bar{R}<R_0$ we have
 	$$ Du\in L^{p^+}(B_{\bar{R}}) $$
 \end{lemma}
\begin{proof}
     $$ \int_{B_{\bar{R}} }|Du|^{p^+}\,dx \le \int_{B_{\bar{R}}}(1+|Du|)^{p^+}\,dx \le \int_{B_{\bar{R}}}(1+|Du|)^{p(x)(1+\delta)}\,dx < +\infty$$
\end{proof}
After these preliminary observations, let us start proving our higher differentiability result.\\
Let us consider $\varphi := u + t v$ for a suitable $v \in W^{1,p(x)}_0(\Omega)$ such that
	\begin{equation}
		\label{cond-v}
		u - \psi + t \, v \ge 0 \qquad \textnormal{for $t \in [0,1)$}.
	\end{equation}
	It is easy to see that such function $\varphi$ belongs to the obstacle class $\mathcal{K}_{\psi}(\Omega)$, because $\varphi = u + t v \ge \psi$.
	\\
	Let us fix a ball $B_{R}$ such that $B_{2R}\Subset \Omega$ with $16R<R_0$ and  a cut off function $\eta\in C^\infty_0(B_R)$, $\eta \equiv 1$ on $B_{\frac{R}{2}}$ such that  $|\nabla \eta|\le \frac{c}{R}$. From now on we will suppose $R\le 1$, without loss of generality due to the local nature of our results.
	Then, for  $|h|<\frac{R}{4}$, we consider
	\begin{equation}
		\label{fun-v}
		v_1(x) = \eta^2(x) [(u - \psi)(x + h) - (u - \psi)(x)].
	\end{equation}
	From the regularity of $u$ and $\psi$, it is immediate to check that $v_1 \in W^{1,p(x)}_0(\Omega)$. Moreover $v_1$ fulfills \eqref{cond-v}.
	\\Indeed, for a.e. $x \in \Omega$ and for any $t \in [0,1)$ we have 
	\begin{eqnarray*}
		&&u(x)- \psi(x)  + t v_1(x)\\
		&=& u(x) - \psi(x)  + t \eta^2(x) [(u - \psi)(x + h) - (u - \psi)(x)]\\
		&=& t \, \eta^2(x) (u - \psi)(x+h) + (1 - t \eta^2(x)) (u - \psi)(x) \ge \, 0,
	\end{eqnarray*}
	because $u \in \mathcal{K}_{\psi}(\Omega)$.
	
	By using $\varphi = u + t v_1$ in \eqref{obst-def} as an admissible test function, with $v_1$ introduced in \eqref{fun-v}, we obtain
	\begin{equation}
		\label{sei}
		0 \le \, \int_{\Omega} \langle \mathcal{A}(x, Du(x)), D[\eta^2(x) [(u - \psi)(x + h) - (u - \psi)(x)]] \rangle \, dx.
	\end{equation}
	On the other hand, if we define
	\begin{equation}
		\label{fun-v-tras}
		v_2(x) = \eta^2(x-h) [(u - \psi)(x-h) - (u - \psi)(x)]
	\end{equation}
	then  $v_2 \in W^{1,p(x)}_0(\Omega)$ and \eqref{cond-v} still is trivially satisfied, due to the fact that
	\begin{eqnarray*}
		&&u(x)- \psi(x)  + t v_2(x)\\
		&=& u(x) - \psi(x)  + t \eta^2(x-h) [(u - \psi)(x-h) - (u - \psi)(x)]\\
		&=& t \, \eta^2(x-h) (u - \psi)(x-h) + (1 - t \eta^2(x-h)) (u - \psi)(x) \ge \, 0.
	\end{eqnarray*}
	Choosing in \eqref{obst-def} as test function $\varphi = u + t v_2$, where $v_2$ is defined in \eqref{fun-v-tras}
	we get 
	$$ 0 \le \, \int_{\Omega} \langle \mathcal{A}(x, Du(x)), D[\eta^2(x-h) [(u - \psi)(x-h) - (u - \psi)(x)]] \rangle \, dx,$$
	By means of a simple change of variable, we obtain
	\begin{equation}
		\label{sette}
		0 \le \, \int_{\Omega} \langle \mathcal{A}(x+h, Du(x+h)), D[\eta^2(x) [(u - \psi)(x) - (u - \psi)(x+h)]] \rangle \, dx.
	\end{equation}
	We can add  \eqref{sei} and \eqref{sette}, thus  obtaining
	\begin{eqnarray*}
		0 &\le& \int_{\Omega} \langle \mathcal{A}(x, Du(x)), D[\eta^2(x) [(u - \psi)(x+h) - (u - \psi)(x)]] \rangle \, dx\\
		&& +  \int_{\Omega} \langle \mathcal{A}(x+h, Du(x+h)), D[\eta^2(x) [(u - \psi)(x) - (u - \psi)(x+h)]] \rangle \, dx\\
		&=& \int_{\Omega} \langle \mathcal{A}(x, Du(x)) - \mathcal{A}(x + h, Du(x+h)), D[\eta^2(x) [(u - \psi)(x+h) - (u - \psi)(x)]] \rangle \, dx,
	\end{eqnarray*}
	which implies, since $D[\eta^2(x) [(u - \psi)(x+h) - (u - \psi)(x)]]=\eta^2(x) D[(u - \psi)(x+h) - (u - \psi)(x)]+ 2 \, \eta(x) \,D \eta(x) \, [(u - \psi)(x+h) - (u - \psi)(x)]$
	\begin{eqnarray*}
		\!\!\!\! 0 &\ge& \int_{\Omega} \langle \mathcal{A}(x + h, Du(x+h)) - \mathcal{A}(x, Du(x)), \eta^2(x) D[(u - \psi)(x+h) - (u - \psi)(x)]\rangle \, dx\\
		&& \!\!\!\!\!\! + \int_{\Omega} \langle \mathcal{A}(x + h, Du(x+h)) - \mathcal{A}(x, Du(x)), 2 \, \eta(x) \, D \eta(x) \, [(u - \psi)(x+h) - (u - \psi)(x)]\rangle \, dx.
	\end{eqnarray*}
	Adding and subtracting the same quantity $\mathcal{A}(x+h, Du(x))$ and by the bilinearity of the scalar product, we can write the previous inequality as follows
	\begin{eqnarray}
		0 &\ge & 
		\int_{\Omega} \langle \mathcal{A}(x+h, Du(x+h)) - \mathcal{A}(x+h, Du(x)), \eta^2 (Du(x+h) - Du(x)) \rangle \, dx \nonumber\\
		&& - \int_{\Omega} \langle \mathcal{A}(x+h, Du(x+h)) - \mathcal{A}(x+h, Du(x)), \eta^2 (D\psi(x+h) - D\psi(x)) \rangle \, dx \nonumber\\
		&& + \int_{\Omega} \langle \mathcal{A}(x+h, Du(x+h)) - \mathcal{A}(x+h, Du(x)), 2 \eta \, D \eta \tau_{h}(u - \psi) \rangle \, dx \nonumber\\
		&& + \int_{\Omega} \langle \mathcal{A}(x+h, Du(x)) - \mathcal{A}(x, Du(x)), \eta^2 (Du(x+h) - Du(x)) \rangle \, dx \nonumber\\
		&& - \int_{\Omega} \langle \mathcal{A}(x+h, Du(x)) - \mathcal{A}(x, Du(x)), \eta^2 (D\psi(x+h) - D\psi(x)) \rangle \, dx \nonumber\\
		&& + \int_{\Omega} \langle \mathcal{A}(x+h, Du(x)) - \mathcal{A}(x, Du(x)), 2 \eta \, D \eta \tau_{h}(u - \psi) \rangle \, dx \nonumber\\
		&= :& I + II + III + IV + V + VI, \label{I-to-VI}\end{eqnarray}
	that leads to
	\begin{equation}\label{otto}
		I \le \, |II| + |III| + |IV| + |V| + |VI|.
	\end{equation}
	
	The $ p(x)$-ellipticity assumption expressed by $(\cA 1)$ implies
	\begin{equation}\label{I}
		I \ge \, \nu \int_{\Omega} \eta^2 |\tau_{h} Du|^2 (1 + |Du(x+h)|^2 + |Du(x)|^2)^{\frac{p(x)-2}{2}} \, dx.
	\end{equation}
	By virtue of assumption $(\cA 2)$ and Young's inequality with exponents $ p=p'=2 $, we get
	\begin{eqnarray}\label{II}
		|II| &\le& L \, \int_{\Omega} \eta^2 |\tau_{h} D u| (1 + |Du(x+h)|^2 + |Du(x)|^2)^{\frac{p(x)-2}{2}} \, |\tau_{h} D \psi| dx\cr\cr
		&\le& \varepsilon \, \int_{\Omega} \eta^2 |\tau_{h} D u|^2 (1 + |Du(x+h)|^2 + |Du(x)|^2)^{\frac{p(x)-2}{2}}  \, dx \cr\cr
		&& + C_{\varepsilon}(L) \int_{\Omega} \eta^2|\tau_{h} D \psi|^2 (1 + |Du(x+h)|^2 + |Du(x)|^2)^{\frac{p(x)-2}{2}}  \, dx \cr\cr
		&\le& \varepsilon \, \int_{\Omega} \eta^2 |\tau_{h} D u|^2 (1 + |Du(x+h)|^2 + |Du(x)|^2)^{\frac{p(x)-2}{2}}  \, dx\cr\cr
		&& + C_{\varepsilon}(L) \left (\int_{B_R} |\tau_{h} D \psi|^{p^+} \, dx \right )^{\frac{2}{p^+}} \, \left (\int_{B_{R}} (1 + |Du(x+h)|^2 + |Du(x)|^2)^{\frac{p(x)-2}{2}\cdot \frac{p^+}{p^+-2}}) \, dx \right )^{\frac{p^+-2}{p^+}}
	\end{eqnarray}
	where we used H\"older's inequality with exponents $\frac{p^+}{2}$ and $\frac{p^+}{p^+-2}$ and the properties of $\eta$.\\
	Here and in the sequel we denote with $\varepsilon>0$ a constant that will be determined later.
	\\ Observing that, since $p(x)<p^+$
	\[
	p^-> 2\,\Longrightarrow\,\,	\frac{p^+}{p^+ - 2} \le \frac{p(x)}{p(x)-2} \,
	\]
	we get
	\begin{eqnarray}\label{IIp}
		&&\int_{B_{R}} (1 + |Du(x+h)|^2 + |Du(x)|^2)^{\frac{p(x)-2}{2}\cdot \frac{p^+}{p^+-2}} \, dx 
		\cr \cr 
		&\le& \int_{B_{R}} (1 + |Du(x+h)|^2 + |Du(x)|^2)^{\frac{p(x)}{2}} \, dx
		\cr \cr 
		&\le& 
		\int_{B_{R}} (1 + |Du(x+h)|^2 + |Du(x)|^2)^{\frac{p^+}{2}} \, dx
		\cr \cr 
		&\le&
		C \int_{B_{R}} (1 + |Du(x+h)|^{p^+} + |Du(x)|^{p^+}) \, dx
		\cr \cr
		& \le&
		C \int_{B_{R}} (1 +|Du(x+h)|^{p^{+}} + |Du(x)|^{p^+}) \, dx
		\cr \cr
		&\le& C \int_{B_{2R}} (1 +|Du(x)|^{p^{+}}) \, dx,
	\end{eqnarray}
	 where we used the second estimate of Lemma \ref{le1}.\\
	 Note that the quantity above is finite due to Lemma \ref{maggint}, applied with $\bar{R}=2R$.\\
	 Inserting \eqref{IIp} in \eqref{II}, using  the assumption $D\psi\in W^{1,\gamma_2}$  and the first estimate of Lemma \ref{le1} in the second integral of the right hand side of \eqref{II}, we obtain
	\begin{eqnarray}\label{IIb}
		|II|
		&\le& \varepsilon \, \int_{\Omega} \eta^2 |\tau_{h} D u|^2 (1 + |Du(x+h)|^2 + |Du(x)|^2)^{\frac{p(x)-2}{2}}  \, dx\cr\cr
		&& + C_{\varepsilon}(L, n, p^{+}) |h|^2\left (\int_{B_{2R}} | D^2 \psi|^{p^+} \, dx \right )^{\frac{2}{p^+}} \, \left (\int_{B_{2R}} (1 + |Du(x)|^{p^+}) \, dx \right )^{\frac{p^+-2}{p^+}}\cr\cr
		&:=&\varepsilon \, \int_{\Omega} \eta^2 |\tau_{h} D u|^2 (1 + |Du(x+h)|^2 + |Du(x)|^2)^{\frac{p(x)-2}{2}}  \, dx+|h|^2 H_2,
	\end{eqnarray}
	where, from now on with $ H_i$, $i=2,3,4,5,6 $ we will denote quantities that are finite in view of our assumptions.
	Arguing analogously, once more by assumption $(\mathcal{A}2)$ we get
	\begin{eqnarray*}
		|III| &\le& 2 L  \int_{\Omega} |\tau_{h} Du| \, (1 + |Du(x+h)|^2 + |Du(x)|^2)^{\frac{p(x)-2}{2}} \, \eta|D \eta| |\tau_{h}(u - \psi)| \, dx\cr\cr
		&\le& \varepsilon \, \int_{\Omega} \eta^2 |\tau_{h} Du|^2 \, (1 + |Du(x+h)|^2 + |Du(x)|^2)^{\frac{p(x)-2}{2}} \, dx \cr\cr
		&& + C_{\varepsilon}(L) \, \int_{\Omega} |D \eta|^2|\tau_{h} (u - \psi)|^2 \, (1 + |Du(x+h)|^2 + |Du(x)|^2)^{\frac{p(x)-2}{2}} dx \cr\cr
		&\le& \varepsilon \, \int_{\Omega} \eta^2 |\tau_{h} Du|^2 \, (1 + |Du(x+h)|^2 + |Du(x)|^2)^{\frac{p(x)-2}{2}} \, dx \cr\cr
		&& + \frac{C_{\varepsilon}(L,p^+)}{R^2} \left (\int_{ B_R }|\tau_{h} (u - \psi)|^{p^+} \, dx \right )^{\frac{2}{p^+}} \left (\int_{B_{2R}} (1 + |Du(x)|^{p^+}) \, dx \right )^{\frac{p^+-2}{p^+}} ,	
	\end{eqnarray*}
	where in the last line we used the same argument for $|II|$ replacing $D\psi$ with $u - \psi$ .\\ Since $u-\psi\in W^{1,p^+}(B_R)$, we may
	use the first estimate of Lemma \ref{le1} to control the last integral in the right hand side of the previous estimate, obtaining that
	\begin{eqnarray}\label{III}
		\!\!\!\!\!\!\!\!|III| &\le &	\varepsilon \, \int_{\Omega} \eta^2 |\tau_{h} Du|^2 \, (1 + |Du(x+h)|^2 + |Du(x)|^2)^{\frac{p(x)-2}{2}} \, dx \cr\cr
		&& + |h|^2\frac{C_{\varepsilon}(L, n, p^+)  }{R^2} \left (\int_{B_{2R} }|D (u - \psi)(x)|^{p^+} \, dx \right )^{\frac{2}{p^+}} \left (\int_{B_{2R}}  (1 + |Du(x)|)^{p^+} \, dx \right )^{\frac{p^+-2}{p^+}}\cr\cr
		&:=&\varepsilon \, \int_{\Omega} \eta^2 |\tau_{h} D u|^2 (1 + |Du(x+h)|^2 + |Du(x)|^2)^{\frac{p(x)-2}{2}}  \, dx+|h|^2 H_3. 
	\end{eqnarray}
	In order to estimate the integral $IV$, we use assumption $(\mathrm{\cA 4}) $, Young's inequality with exponents $p=p'=2$ and H\"older's inequality with exponents $\frac{n}{2}$ and $\frac{n}{n-2}$ and the fact that $$\log(e+|Du(x)|^2)\le\log((e+|Du(x)|)^2) = 2\log(e+|Du(x)|)$$
	\begin{eqnarray}\label{IV}
		|IV| &\le& |h| \int_{\Omega} \eta^2\, |\tau_{h} Du| (\kappa(x+h)+\kappa(x)) (1 + |Du(x)|^2)^{\frac{p(x)-1}{2}}  \,\log(e+|Du(x)|^2)  dx\cr\cr
		&=& |h| \int_{\Omega} \eta^2\, |\tau_{h} Du| (\kappa(x+h)+\kappa(x)) (1 + |Du(x)|^2)^{\frac{p(x)-2}{4}} (1 + |Du(x)|^2)^{\frac{p(x)}{4}} \,\log(e+|Du(x)|^2)  dx\cr\cr
		&\le& \varepsilon \int_{\Omega} \eta^2|\tau_{h} Du|^2 \,  \left (1 + |Du(x)|^2 + |Du(x+h)|^2  \right )^{\frac{p(x)-2}{2}} \, dx\cr\cr
		&&  + C_{\varepsilon} |h|^{2} \int_{{B_R}} (\kappa(x+h)+\kappa(x))^2 (1 + |Du(x)|^2)^{\frac{p(x)}{2}} \,\,\log^2(e+|Du(x)|) dx \cr\cr
		&\le& \varepsilon \int_{\Omega} \eta^2 |\tau_{h} Du|^2 \,  \left (1 + |Du(x)|^2 + |Du(x+h)|^2  \right )^{\frac{p(x)-2}{2}} \, dx \cr\cr
		&& + C_{\varepsilon} |h|^{2} \left (\int_{B_R} (\kappa(x+h)+\kappa(x))^n \,\log^n(e+|Du(x)|) \, dx \right )^{\frac{2}{n}} \left ( \int_{{B_R}} (1 + |Du(x)|^2)^{\frac{np(x)}{2(n - 2)}} \, dx \right )^{\frac{n - 2}{n}}
	\end{eqnarray}
	where we also used that $\mathrm{supp}\eta\subset B_R$. \\ \\
		Let us consider now the term
	\begin{eqnarray} \label{klog}
	&&\int_{B_R} (\kappa(x+h)+\kappa(x))^n\log^n(e+|Du(x)|) \, dx \cr\cr
	&\le& C\int_{B_{R} }(\kappa^n(x+h)+\kappa^n(x))\log^n(e+|Du(x)|) \, dx \cr\cr
	&=& C\int_{B_{R} }\kappa^n(x)\log^n(e+|Du(x)|) \, dx + C\int_{B_{R} }\kappa^n(x+h)\log^n(e+|Du(x)|) \, dx.
	\end{eqnarray}
	Now we use Lemma \ref{dislog} for the first integral in the right hand side of \eqref{klog} with
    $$\varepsilon=1 \quad \gamma=1 \quad \alpha=n \quad s=\kappa^n(x) \quad t=\log^n(e+|Du(x)|) $$ 
    then we get
    \begin{eqnarray} \label{kh1}
    	&&\int_{B_{R} }\kappa^n(x)\log^n(e+|Du(x)|) \, dx \cr\cr &\le& \int_{B_{R} }\kappa^n(x)\log^n(e+k^n(x)) \, dx + \int_{B_{R} }(e+|Du(x)|-1)\log^n(e+|Du(x)|) \, dx \cr\cr 
    	&\le& C\int_{B_{R} }\kappa^n(x)\log^n(e+k(x)) \, dx + \int_{B_{R} }(e+|Du(x)|-1)\log^n(e+|Du(x)|) \, dx.
    \end{eqnarray}
    Again we use Lemma \ref{dislog} for the second integral in the right hand side of \eqref{klog}  with $$\varepsilon=1 \quad \gamma=1 \quad \alpha=n \quad s=\kappa^n(x+h) \quad t=\log^n(e+|Du(x)|) $$ 
    then we get
    \begin{eqnarray} \label{kh2}
    	&&\int_{B_{R} }\kappa^n(x+h)\log^n(e+|Du(x)|) \, dx \cr\cr &\le& \int_{B_{R} }\kappa^n(x+h)\log^n(e+k^n(x+h)) \, dx + \int_{B_{R} }(e+|Du(x)|-1)\log^n(e+|Du(x)|) \, dx \cr\cr 
    	&\le& C\int_{B_{R} }\kappa^n(x+h)\log^n(e+k(x+h)) \, dx + \int_{B_{R} }(e+|Du(x)|-1)\log^n(e+|Du(x)|) \, dx
    	\cr\cr 
    	&\le& C\int_{B_{2R} }\kappa^n(x)\log^n(e+k(x)) \, dx + \int_{B_{R} }(e+|Du(x)|-1)\log^n(e+|Du(x)|) \, dx.
    \end{eqnarray}
    Inserting \eqref{kh1} and \eqref{kh2} in \eqref{klog} we have 
    \begin{eqnarray} \label{log41}
    	&&\int_{B_R} (\kappa(x+h)+\kappa(x))^n\log^n(e+|Du(x)|) \, dx \cr\cr
    	&\le&  C\int_{B_{2R} }\kappa^n(x)\log^n(e+k(x)) \, dx + C\int_{B_{R} }(e+|Du(x)|-1)\log^n(e+|Du(x)|) \, dx
    \end{eqnarray}
    
    where the first integral in the right hand side is finite since $\kappa \in L^n\log^n L$ by assumption $(\mathcal{A}4)$.\\
    For the second integral in the right hand side of \eqref{log41} we have 
    \begin{eqnarray} \label{log42}
    	&& \int_{B_R} (e-1 + |Du(x)|)\log^n(e+|Du(x)|)\,dx \cr \cr
    	&\le& C\int_{B_R \cap \{|Du| \ge \, e\}} |Du(x)| \log^n (e + |Du(x)|) \, dx + C \, R^n\cr\cr
    	&\le& C\int_{B_R \cap \{|Du| \ge \, e\}} |Du(x)| \log^n (e+\frac{|Du(x)|}{\| |Du| \|_1} \cdot \| |Du| \|_1) \, dx + C \, R^n\cr\cr
    	&=& C \, R^n \fint_{B_R} |Du(x)| \log^n \left (e + \frac{|Du(x)|}{\| |Du| \|_1} \right ) \, dx + C\, R^n \, \fint_{B_R} |Du(x)| \log^n (e + \| |Du| \|_1) \, dx + C \, R^n\cr\cr
    	&\le&  C \, R^n\left (\fint_{B_R} |Du(x)|^{2} \, dx\right )^{\frac{1}{2}} + C \, R^n {\fint_{B_R}} |Du(x)|\log^n (e + {\| |Du | \|_1}) \, dx + C \, R^n, 
    \end{eqnarray}
    where we used the elementary fact
    $$\log(e+ab)\le\log(e+a)+\log(e+b) \quad\quad \forall \, a,b>0$$
    and Lemma \ref{normDu}.
    Observe that the quantities in the right hand side are finite, since $ u\in W^{1,p(x)}$ with $p(x)>2$.\\

    Coming back to \eqref{IV}, we aim to estimate the term $$\left ( \int_{{B_R}} (1 + |Du(x)|^2)^{\frac{np(x)}{2(n - 2)}} \, dx \right )^{\frac{n - 2}{n}},$$
    this is the crucial point where the Calder\'on-Zygmund result is used. We observe that for instance in the paper \cite{Gav1} this problem has been overcome by means of the results in \cite{EPdN18} which provides the necessary higher integrability result so this term turns to be bounded.\\
	Since $D\psi\in W^{1,p^+}(\Omega)$, classical Sobolev embedding Theorem implies $D\psi\in L^{(p^+)^*}(\Omega)$. \\
	Observing that
	\[
	\gamma_2\ge 2\,\,\Longrightarrow\,\,\frac{np(x)}{n - 2} \le \frac{np^+}{n - p^+}=(p^+)^*  ,
	\] \\
    we have $|D\psi|^{p(x)}\in L^{\frac{n}{n-2}}_{\loc}(\Omega)$\\
	 Therefore, Theorem \ref{habermann10} with $q=\frac{n}{n-2} $ implies $|Du|^{p(x)}\in L^\frac{n}{n-2}(\Omega)$, in particular there holds:
	\begin{eqnarray}\label{highint}
		 \left[\fint_{B_R}|Du|^{\frac{n p(x)}{n-2}}\,dx\right]^{\frac{n-2}{n}}\le C K^\delta \fint_{B_{8R}}|Du|^{p(x)}\,dx + C K^\delta 
		 \left[\fint_{B_{8R}}|D\psi|^{\frac{np(x)}{n-2}}\,dx+1\right]^{\frac{n-2}{n}}
	\end{eqnarray}
where $c=c(n,\gamma_1,\gamma_2,\nu,L)$
and
$K:=\int_{B_{8R}}(|Du|^{p(x)}+|D\psi|^{p(x)(1+\delta)})\,dx+1 $\\
Where we choose $\varepsilon=\delta$ of the higher integrability since it is not restrictive to suppose $1+\delta<q=\frac{n}{n-2}$.\\
Observe that we can replace $K$ with $$M:=\int_{\Omega}(|Du|^{p(x)}+|D\psi|^{\gamma_2^*})\,dx<+\infty $$ since $p(x)(1+\delta)<(p^+)^*\le\gamma_2^*.$\\
We can write the previous estimate as follows: 
    \begin{eqnarray} \label{K4}
    	\left ( \int_{{B_R}} (1 + |Du(x)|^2)^{\frac{np(x)}{2(n - 2)}} \, dx \right )^{\frac{n - 2}{n}}\le \frac{C M^{(p^+)^*}}{R^2}\int_{B_{8R}}|Du|^{p(x)}dx+CM^{(p^+)^*}\left[\int_{B_{8R}}|D\psi|^{(p^+)^*}+1\,dx\right]^\frac{n-2}{n}
	\end{eqnarray}

Inserting \eqref{log41}, \eqref{log42} and \eqref{K4} in \eqref{IV} we get
\begin{eqnarray}\label{IVbis}
	 |IV|
	 &\le&\varepsilon \int_{\Omega} \eta^2 |\tau_{h} Du|^2 \,  \left (1 + |Du(x)|^2 + |Du(x+h)|^2  \right )^{\frac{p(x)-2}{2}} \, dx \cr\cr
	 &&+C_\varepsilon |h|^2 \left \{\left(\int_{B_{2R} }\kappa^n(x)\log^n(e+k(x)) \, dx\right)^{\frac{2}{n}}+ C \,R^2\left[  \left (\fint_{B_R} |Du|^2 \, dx\right )^{\frac{1}{2}} + {\fint_{B_R}} |Du|\log^n (e + {\| |Du | \|_1}) \, dx + 1  \right]^\frac{2}{n}\right\}\cr\cr
	 &&\cdot \left\{\frac{C M^{(p^+)^*}}{R^2}\int_{B_{8R}}|Du|^{p(x)}dx+CM^{(p^+)^*}\left[\int_{B_{8R}}|D\psi|^{(p^+)^*}+1\,dx\right]^\frac{n-2}{n}\right\}\cr\cr
	&:=&\varepsilon \int_{\Omega} \eta^2 |\tau_{h} Du|^2 \,  \left (1 + |Du(x)|^2 + |Du(x+h)|^2  \right )^{\frac{p(x)-2}{2}} \, dx +|h|^2 H_4 
\end{eqnarray}
	
	Assumption $(\mathrm{\cA 4}) $ also entails
	\begin{eqnarray}\label{V}
		|V| &\le& |h| \int_{\Omega} \eta^2 \, |\tau_{h} D\psi(x)|(\kappa(x+h)+\kappa(x)) (1 + |Du(x)|^2)^{\frac{p(x)-1}{2}} \,\log(e+|Du(x)|^2)  dx\cr\cr
		&\le& C|h| \left (\int_{B_R} |\tau_{h} D\psi|^{p^+} \, dx \right)^{\frac{1}{p^+}} \left (\int_{B_R}  (\kappa(x+h)+\kappa(x))^{\frac{p^+}{p^+-1}} \,(\log(e+|Du(x)|))^\frac{p^+}{p^+-1} (1 + |Du(x)|^2)^{\frac{p(x)-1}{2}\cdot \frac{p^+}{p^+-1}}  \, dx \right )^{\frac{p^+-1}{p^+}} \cr\cr
		&\le& C|h| \left (\int_{B_R} |\tau_{h} D\psi|^{p^+} \, dx \right)^{\frac{1}{p^+}} \left (\int_{B_R}  (\kappa(x+h)+\kappa(x))^{\frac{p^+}{p^+-1}}\,(\log(e+|Du(x)|))^\frac{p^+}{p^+-1} \, (1 + |Du(x)|^2)^{\frac{p(x)}{2}}  \, dx \right )^{\frac{p^+-1}{p^+}} \cr\cr
		&\le& C|h| \left (\int_{B_R} |\tau_{h} D\psi|^{p^+} \, dx \right)^{\frac{1}{p^+}} \left (\int_{B_R} (\kappa(x+h)+\kappa(x))^n\,\log^n(e+|Du(x)|) \, dx \right )^{\frac{1}{n}} \cr\cr
		&&\qquad\quad \cdot \left (\int_{B_R} (1 + |Du(x)|)^{\frac{p(x)n(p^+-1)}{n(p^+-1) - p^+}} \, dx \right )^{\frac{n(p^+-1)- p^+}{np^+}} 
	\end{eqnarray}
	where we used the properties of $\eta$ and H\"older's inequality with exponents $p^+$ and $\frac{p^+}{p^+-1}$, and again H\"older's inequality with exponents $\frac{n(p^+-1)}{p^+}$ and $\frac{n(p^+-1)}{n(p^+-1)-p^+}$. Note that these exponents are greater than $1$ since $p^+<n$ and $p^+>2.$ \\
	Observe that the second integral in the right hand side is the same as the one in the previous estimate and therefore it can be estimated with \eqref{log42}.\\
	On the other hand, since  $|D\psi|\in L^{(p^+)^*}(\Omega)$ and since
	$$ \frac{np(x)(p^+-1)}{n(p^+-1) - p^+}\le \frac{np^+}{n - p^+}=(p^+)^*\quad \Longleftrightarrow\quad p^+\ge 2$$
	we have $ |D\psi|^{p(x)}\in L^\frac{n(p^+-1)}{n(p^+-1) - p^+}_{\loc}(\Omega)$\\ 
	Therefore, Theorem \ref{habermann10} with exponent $q=\frac{n(p^+-1)}{n(p^+-1) - p^+}$ implies $|Du|^{p(x)}\in L^\frac{n(p^+-1)}{n(p^+-1) - p^+}(\Omega)$, in particular there holds:
	\begin{eqnarray}\label{highint}
		\left[\fint_{B_R}|Du(x)|^{\frac{n p(x)(p^+-1)}{n(p^+-1)-p^+}}\,dx\right]^{\frac{n(p^+-1)-p^+}{n(p^+-1)}}\le C K^\delta \fint_{B_{8R}}|Du(x)|^{p(x)}\,dx + C K^\delta 
		\left[\fint_{B_{8R}}|D\psi(x)|^{\frac{n p(x)(p^+-1)}{n(p^+-1)-p^+}}\,dx+1\right]^{\frac{n(p^+-1)-p^+}{n(p^+-1)}}
	\end{eqnarray}
	where $c=c(n,\gamma_1,\gamma_2,\nu,L)$
	and
	$K:=\int_{B_{8R}}(|Du(x)|^{p(x)}+|D\psi(x)|^{p(x)(1+\delta)})\,dx+1 $\\
	Arguing as in \eqref{K4} we can replace $K$ with $$M:=\int_{\Omega}(|Du(x)|^{p(x)}+|D\psi(x)|^{\gamma_2^*})\,dx<+\infty$$ since $p(x)(1+\delta)<(p^+)^*<\gamma_2^*$.\\
	We can write the previous estimate as follows: 
	\begin{eqnarray}
		&&\left[\int_{B_R}|Du(x)|^{\frac{n p(x)(p^+-1)}{n(p^+-1)-p^+}}\,dx\right]^{\frac{n(p^+-1)-p^+}{n(p^+-1)}}\cr\cr
		&\le& \frac{C M^{\frac{p^+}{n(p^+-1)-p^+}}}{R^\frac{p^+}{p^+-1}} \int_{B_{8R}}|Du(x)|^{p(x)}\,dx + C M^{\frac{p^+}{n(p^+-1)-p^+}}
		\left[\int_{B_{8R}}|D\psi(x)|^{(p^+)^*}\,dx+R^n\right]^{\frac{n(p^+-1)-p^+}{n(p^+-1)}}
	\end{eqnarray}
	Then we get 
	\begin{eqnarray} \label{K5}
		&&\left (\int_{B_R} (1 + |Du(x)|)^{\frac{p(x)n(p^+-1)}{n(p^+-1) - p^+}} \, dx \right )^{\frac{n(p^+-1)- p^+}{np^+}}\cr\cr
		&\le& 
		\frac{C M^{(p^+)^*}}{R} \left[\int_{B_{8R}}|Du(x)|^{p(x)}\,dx\right]^\frac{p^+-1}{p^+} + C M^{(p^+)^*}
		\left[\int_{B_{8R}}|D\psi(x)|^{(p^+)^*}\,dx+R^n\right]^{\frac{n(p^+-1)-p^+}{np^+}}
	\end{eqnarray}

	Finally, by virtue of the assumption $D\psi\in W^{1,p^+}(\Omega)$, we can use the first inequality of Lemma \ref{le1} to estimate   the first integral in the right hand side of \eqref{V}, and inserting \eqref{log41}, \eqref{log42} and \eqref{K5} in \eqref{V} we get\\
	\begin{eqnarray}\label{Vbis}
		|V|
		&\le& C|h|^2 \, \left (\int_{B_{2R}} |D^2\psi(x)|^{p^+} \, dx \right)^{\frac{1}{p^+}}\cr\cr
	    &&\cdot \left\{\left (\int_{B_{2R} }\kappa^n(x)\,\log^n(e+k(x)) \, dx \right )^{\frac{1}{n}} + C \,R\left[  \left (\fint_{B_R} |Du(x)|^2 \, dx\right )^{\frac{1}{2}} + {\fint_{B_R}} |Du(x)|\log^n (e + {\| |Du | \|_1}) \, dx + 1  \right]^\frac{1}{n}\right\} \cr\cr
	    &&\cdot \left\{\frac{C M^{(p^+)^*}}{R} \left[\int_{B_{8R}}|Du(x)|^{p(x)}\,dx\right]^\frac{p^+-1}{p^+} + C M^{(p^+)^*}
	    \left[\int_{B_{8R}}|D\psi(x)|^{(p^+)^*}\,dx+1\right]^{\frac{n(p^+-1)-p^+}{np^+}}\right\}\cr\cr
	    &:=&|h|^2H_5
	\end{eqnarray}
	with $C=C(n, \nu, L, \ell, p^+, R)$.
	\\
	Finally, arguing as we did for the estimate of $V$ we get
	\begin{eqnarray*}
		|VI| &\le& 2|h| \, \int_{\Omega} \eta |D\eta|\, |\tau_{h} (u - \psi)(x)| (\kappa(x+h)+\kappa(x)) (\mu^2 + |Du(x)|^2)^{\frac{p(x)-1}{2}} \,\log(e+|Du(x)|^2) dx\\
		&\le& \frac{C|h|}{R} \, \left (\int_{B_R} |\tau_{h} (u - \psi)(x)|^{p^+} \, \, dx \right)^{\frac{1}{p^+}} \left (\int_{B_R}  (\kappa(x+h)+\kappa(x))^{\frac{p^+}{p^+-1}}(\log(e+|Du(x)|))^\frac{p^+}{p^+-1} \, (1 + |Du(x)|^2)^{\frac{p(x)}{2}}  \, dx \right )^{\frac{p^+-1}{p^+}} \\
		&\le& \frac{C \, |h|}{R} \left (\int_{B_R} |\tau_{h} (u - \psi)(x)|^{p^+} \, dx \right)^{\frac{1}{p^+}} \cdot \left (\int_{B_R}  (\kappa(x+h) + \kappa(x))^n \,\log^n(e+|Du(x)|) dx \right )^{\frac{1}{n} } \cr\cr
		&&\qquad\quad \cdot \left (\int_{B_R} (1 + |Du(x)|)^{\frac{p(x)n(p^+-1)}{n(p^+-1) - p^+}} \, dx \right )^{\frac{n(p^+-1)- p^+}{np^+}} 		
	\end{eqnarray*}
	where we used the fact that $|D\eta|\le \frac{C}{R}$.
	Using  Lemma \ref{le1} and inserting \eqref{log41}, \eqref{log42} and \eqref{K5} in the previous estimate, we get
	\begin{eqnarray}\label{VI}
		|VI| &\le& C |h|^{2} \left (\int_{B_{2R}} |D (u - \psi)(x)|^{p^+} \, dx \right)^{\frac{1}{p^+}} \cr\cr	
		 &&\cdot \left\{\left (\int_{B_{2R} }\kappa^n(x)\,\log^n(e+k(x)) \, dx \right )^{\frac{1}{n}} + C \,R\left[  \left (\fint_{B_R} |Du(x)|^2 \, dx\right )^{\frac{1}{2}} + {\fint_{B_R}} |Du(x)|\log^n (e + {\| |Du | \|_1}) \, dx + 1  \right]^\frac{1}{n}\right\} \cr\cr
		&&\cdot \left\{\frac{C M^{(p^+)^*}}{R} \left[\int_{B_{8R}}|Du(x)|^{p(x)}\,dx\right]^\frac{p^+-1}{p^+} + C M^{(p^+)^*}
		\left[\int_{B_{8R}}|D\psi(x)|^{(p^+)^*}\,dx+1\right]^{\frac{n(p^+-1)-p^+}{np^+}}\right\}\cr\cr
		&:=&|h|^2H_6	
	\end{eqnarray}
	\\
	Inserting  estimates \eqref{I}, \eqref{IIb}, \eqref{III}, \eqref{IVbis}, \eqref{Vbis} and \eqref{VI} in \eqref{otto}, we infer the existence of  constants $C_\varepsilon \equiv C_\varepsilon(\varepsilon,\nu, L, \ell, n, \gamma_1,\gamma_2, R)$ and $C \equiv C(\nu, L, \ell, n, \gamma_1,\gamma_2, R)$ such that
	\begin{eqnarray*}
		&& \nu \int_{\Omega} \eta^2 |\tau_{h} Du(x)|^2 (1 + |Du(x+h)|^2 + |Du(x)|^2|)^{\frac{p(x)-2}{2}} \, dx\cr\cr
		&\le &\, 3 \varepsilon \int_{\Omega} \eta^2 |\tau_{h} Du(x)|^2 (1 + |Du(x+h)|^2 + |Du(x)|^2|)^{\frac{p(x)-2}{2}} \, dx \cr\cr
		&&+|h|^2(H_2+H_3+H_4+H_5+H_6)
	\end{eqnarray*}
	Choosing $\varepsilon = \frac{\nu}{6}$ we get
	\begin{eqnarray*}
		&& \nu \int_{\Omega} \eta^2 |\tau_{h} Du(x)|^2 (1 + |Du(x+h)|^2 + |Du(x)|^2|)^{\frac{p(x)-2}{2}} \, dx \le |h|^2H,
	\end{eqnarray*}
    where $H:=H_2+H_3+H_4+H_5+H_6$ is a finite quantity.\\
	Using Lemma \ref{Vi} in the left hand side of previous estimate and recalling that $\eta\equiv 1$ on $B_{\frac{R}{2}}$,  we get
	\begin{eqnarray} \label{stimafinale}
	    &&\nu \int_{\Omega} \eta^2 |\tau_{h} Du(x)|^2 (1 + |Du(x+h)|^2 + |Du(x)|^2|)^{\frac{p(x)-2}{2}} \, dx\cr\cr
	    &\ge&\nu \int_{\Omega} \eta^2 |\tau_{h} Du(x)|^2 (1 + |Du(x+h)|^2 + |Du(x)|^2|)^{\frac{\gamma_1-2}{2}} \, dx\cr\cr
	    &\ge& \nu \int_{\Omega} \eta ^2 |\tau_{h}V_{\gamma_1}( Du(x))|^2 \, dx\cr\cr
	    &\ge& \nu \int_{B_{\frac{R}{2}}} |\tau_{h}V_{\gamma_1}( Du(x))|^2 \, dx
    \end{eqnarray}
    Then we have $$\int_{B_{\frac{R}{2}}} |\tau_{h}V_{\gamma_1}( Du(x))|^2 \, dx\le H|h|^2$$
	Lemma \ref{lep} implies that $V_{\gamma_1}(Du)\in W^{1,2}(B_{\frac{R}{2}})$
	and the conclusion follows recalling  the definition of  $V_{\gamma_1}(\xi)$ in \eqref{Vi}.

\end{document}